\documentclass{article}

\newcommand{\titel}
{Equational Axioms for Expected Value Operators\\
}

\newcommand{\midd}{\:|\:}

\usepackage{stmaryrd,amssymb,amsmath,amsthm,url}
\usepackage{graphicx,version}
\usepackage{xfrac}
\usepackage[usenames]{color}
\usepackage{wrapfig}
\newtheorem{theorem}{Theorem}
  
\newtheorem{proposition}{Proposition}

\newtheorem{problem}{Problem} 
  
\newtheorem{definition}{Definition}  
\theoremstyle{definition}

\newcommand{\sg}{\operatorname{{\mathbf s}}}
\newcommand{\reals}{\ensuremath{\mathbb{R}}}
\newcommand{\rats}{\ensuremath{\mathbb{Q}}}
\newcommand{\nats}{\ensuremath{\mathbb{N}}}
\newcommand{\events}{\ensuremath{\mathbb{E}}}
\newcommand{\pf}{\ensuremath{\widehat{P}}}
\newcommand{\mead}{\ensuremath{\mathbb{M}}}
\newcommand{\cv}{\ensuremath{\mathbb{CV}}}

\newcommand{\Md}{\ensuremath{\mathsf{Md}}}
\newcommand{\BA}{\ensuremath{\mathsf{BA}}}
\newcommand{\Sign}{\ensuremath{\mathsf{Sign}}}
\newcommand{\PFBC}{\ensuremath{\mathsf{PFBC}}}
\newcommand{\PFA}{\ensuremath{\mathsf{PFA}}}
\newcommand{\PFF}{\ensuremath{\mathsf{PFF}}}

\newcommand{\DO}{\ensuremath{\mathsf{DO}}}
\newcommand{\ABS}{\ensuremath{\mathsf{ABS}}}
\newcommand{\EPV}{\ensuremath{\mathsf{EPV}}}
\newcommand{\EV}{\ensuremath{\mathsf{EV}}}

\newcommand{\ECV}{\ensuremath{\mathsf{ECV}}}
\newcommand{\EPCV}{\ensuremath{\mathsf{EPCV}}}

\newcommand{\UCV}{\ensuremath{\mathsf{UCV}}}
\newcommand{\Cond}{\ensuremath{\mathsf{Cond}}}
\newcommand{\MBPC}{\ensuremath{\mathsf{MBPC}}}
\title{\titel}

\author{
	Jan A.\ Bergstra 
	\\
\\
  {
	  Informatics Institute,
	  University of Amsterdam}\footnote{\small Email: \texttt{j.a.bergstra@uva.nl, janaldertb@gmail.com}. This is a 
	  significantly rewritten and improved version under a new title of a previous report with title 
	  ``Conditional values in signed meadow based probability calculus'' (\texttt{https://arxiv.org/abs/1609.02812v3}).
	}
}
\date{\small{May 25, 2019}}

\begin{document}
\maketitle

\begin{abstract}
An equational axiomatisation of probability functions for one-dimensional event spaces 
in the language of signed meadows is expanded with conditional values.  
Conditional values constitute a so-called signed vector meadow.
In the presence of a probability function, 
equational axioms are provided for expected value, variance, covariance, and correlation squared, each defined
 for conditional values.

Finite support summation is introduced as a binding operator on meadows which simplifies formulating 
requirements on probability mass functions with finite support. 
Conditional values are related to probability mass functions and to random variables. 
The definitions are reconsidered in a finite dimensional setting.
~\\[4mm]
\end{abstract}

\section{Introduction}
In~\cite{BP2016} a proposal is made for a loose algebraic specification  probability functions in the context of signed meadows. 
The objective of this paper is to proceed on the basis of  the results of~\cite{BP2016} and to provide an 
account of some basic elements of probability calculus including  
probability mass functions,  probability functions, expected value operators, variance, covariance, 
correlation, independence, sample space, and random variable. 
 Ample use is made of the special properties of meadows, most notably $1/0 = 0$, 
and the presentation optimises the match with the equational and axiomatic setting of meadows. 

A conventional ordering of the introduction of concepts in probability theory is as follows: 
(i) given a sample space $S$, an event space $E$ is introduced as a subset of the power set of $S$. 
Then (ii) probability functions are defined over event spaces and (iii) discrete random variables are introduced 
as real functions on $S$ with a countable support. Given these ingredients, (iv) expected value and variance 
are defined for discrete random variables and 
covariance, and correlation  are defined for pairs of random variables. (v) Subsequently probability 
mass functions are derived from 
random variables, (vi) multivariate discrete random variables are introduced as vectors of random variables on the same event space, 
and (vii) joint probability mass functions are derived from multivariate random variables, 
(viii) independence is defined for joint probability mass functions, 
and (ix) marginalisation is defined as a transformation on joint probability mass functions. 
(x) The development of concepts and definitions is redone for the continuous 
case with probability distributions replacing probability mass functions and general 
random variables replacing discrete random variables.

Below the same topics are discussed, though under restrictive conditions, 
and in a different order. A central role is played by probability mass functions with finite support. 
These are meadow valued functions taking nonzero values for finitely many arguments,
and such that the sum of the nonzero values adds up to one. As a conceptual cornerstone the 
space (sort) of conditional values is introduced.

As a consequence of these  choices the definition for the expected value operator, 
and the definitions  of (co)variance and correlation which directly depend on expected values, 
will be repeated in three different settings: 
(i) for probability mass functions with finite support, (ii) for an event space equipped with a sort of conditional values  and a
probability function, and (iii) for a multidimensional event space equipped with a sort of conditional values and a family 
of multivariate probability functions.

With this order of presentation  an adequate match is obtained with 
meadow based equational axiomatisations. The account of probability mass functions is independent of probability theory. 
By defining expected values and derived quantities on conditional values over an 
event structure the incentive for introducing a sample space is avoided, thus avoiding the introduction of 
a proper subsort of samples for the sort of events, and thereby
maintaining the simplicity of the use of a loose equational  specification for probability functions.

\subsection{Survey of the paper}
In Section~\ref{MVSM} meadows are discussed and so-called signed vector meadows are introduced. A novel binding operator, 
called finite support summation (FSS) is introduced and examples of its use are provided.

In Section~\ref{PMFswfs} the notion of a probability mass function (PMF) with finite support is introduced and its formal specification
in the setting of meadows is provided with the help of finite support summation. By default 
a PMF is assumed to be univariate. 

Marginalisation is defined as a family of transformations from a PMF with 
more than one argument (i.e. a multivariate PMF) to a PMF with a smaller number of arguments. 
Expected value and variance are defined as  functionals on univariate PMFs and 
covariance and correlation are defined as functionals on bivariate PMFs.

Having developed an account of PMFs independently of axioms for probability functions, Section~\ref{PFsCVs} 
proceeds with a recall from~\cite{BP2016} of the combination of an event space (a Boolean algebra) and value space (a meadow), 
and the equational specification of a probability function. 
Two versions of Bayes' rule are considered and the relative position of these statements w.r.t. the various axioms is examined.

A conditional operator is applied to events 
and the results of the operator are collected in an additional sort $V_C$ of so-called conditional values (CVs), which constitutes a
so-called  finite dimensional vector space  meadow. 

Thinking in terms of outcomes of a probabilistic process one may assume that the process produces as an outcome an entity of some sort. 
Events from an event space $\events$ represent assessments about the outcome. It is plausible that besides Boolean assessments also 
values, for instance rationals or reals, are considered attributes of an outcome. A CV directly relates values  to events.  
In the presence of a probability function  two equations specify the expected value of a CV.

From~\cite{BP2016}  the specification of probability function families relative to an arity family  is imported, and in Section~\ref{MDC}
an corresponding axiomatisation for expected value operators is provided for the finite dimensional case.

According to~\cite{BP2016} the equations of $\BA + \Md + \Sign + \ABS+ \PFBC_{\!P}+ \PFA_P$ 
constitute a finite equational basis for the class of Boolean algebra based, real valued probability functions, and the
 proof theoretic results, viz. soundness and completeness, concerning signed meadows 
of~\cite{BBP2013, BBP2015} extend to the case 
with Boolean algebra based probability functions. The axiom system $\BA + \Md + \Sign + \ABS+ \PFBC_{\!P}+ \PFA_P$ 
is merely a particular formalisation of Kolmogorov's 
axioms for probability theory phrased in the context of  meadows 
and the completeness result asserts the completeness of this particular formalisation w.r.t. its standard model. 
The main result of the paper is to provide an extension of this axiomatisation 
 with conditional values  and expected value operators $E_P$.

\subsection{On the  use of equational logic}
By working in first order  equational logic, 
I am able to make use of an axiomatic style when developing basic elements of
the  theory of probability. The exposition therewith is somewhat biased towards formalisation.

We will develop an incremental collection of equational specifications. These specifications can be understood in three different ways:
(a) as formalisation of a preferred underlying mathematical reality, (b) as axioms determining what is claimed with one or more underlying mathematical realities primarily as a principled method  for justification, (c) as a combination of (a) and (b) in which both views (a) 
and (b) have been assimilated, thus following the terminology of~\cite{NicaudBG2001}, in such a manner that when appropriate, 
or when required for 
clarification explicit dis-assimilation of both views (again following~\cite{NicaudBG2001}) is an option.

The objective of formalisation and axiomatisation in this paper is not inherited
from an overarching intention to avoid mistakes, as is the most prominent rationale for formalisation  in computer science.
Instead the objective is to use the axiomatic approach to obtain a uniform degree of  clarity about assumptions, working hypotheses, 
patterns of reasoning, and patterns of calculation. 

I will not distinguish between names for constants and functions of meadows and their mathematical counterparts. 
Rather than writing say $\reals_0 \models t= r$,  in cases where 
 ordinary  mathematics suggests writing $t= r$ and provided there is no risk of confusion, that is a preferred arithmetical datatype is supposed to be known to the reader, ``$t=r$'' is preferred. 
 On the other hand sort names, e.g. $E$ for events, will be distinguished from the corresponding carriers, (e.g.
 $|\events|$ in the case of $E$) and a 
 specific probability function with intended to serve as an interpretation of $P$ will be referred to as $\pf$.

In summary: below equational logic is applied with the following objectives in mind:  
(i) to demonstrate that an axiomatic approach in terms of equational logic 
to elementary probability calculus is both feasible and attractive, 
(ii) to illustrate  the compatibility of an axiomatic approach to probability calculus  
with  conventional mathematical style and notation, and 
(iii) to provide optimal clarity about the assumptions which underly the various definitions, 
while (iv) using meadows as a tool throughout the presentation. 

\subsection{Assimilation and dis-assimilation: a plurality of options}
In~\cite{} the idea is put forward that persons working with mathematical text may or may not at a certain stage 
assimilate, that is perceive as one and the same, different notations. For instance it is common not to distinguish $1$ and $+1$, but
one may well imagine a stage in which these notations are distinguished. I mentioned already that not making a distinction between 
constants $0$ and $1$ and the corresponding values in a chosen mathematical domain may be understood as a result of assimilation
which may be reversed temporarily by way of implicit or explicit dis-assimilation.

In Paragraph~\ref{MofCV} below so-called conditional values are introduced as elements of a sort $V_C$. 
These are understood as a version of the numbers, 
though with conditions, the connection between conditional values and (ordinary, i.e. unconditional) 
values is made via an embedding $v \colon V \to CV$ 
Here it is an option to assimilate conditional values and unconditional values and to drop occurrences of $v(\_)$. 
I have chosen not to do so, and to leave unconditional values and conditional values dis-assimilated in order to allow a better
focus on the fact that the conditional values constitute a vector space meadow rather than a meadow.

\begin{table}
\centering
\hrule
\begin{align}
\label{eq:7}
	(x+y)+z &= x + (y + z)\\
\label{eq:8}
	x+y     &= y+x\\
\label{eq:9}
	x+0     &= x\\
\label{eq:10}
	x+(-x)  &= 0\\
\label{eq:11}
	(x \cdot y) \cdot  z &= x \cdot  (y \cdot  z)\\
\label{eq:12}
	x \cdot  y &= y \cdot  x\\
\label{eq:13}
	1\cdot x &= x \\
\label{eq:14}
	x \cdot  (y + z) &= x \cdot  y + x \cdot  z\\
\label{eq:15}
	(x^{-1})^{-1} &= x \\
\label{eq:16}
	x \cdot (x \cdot x^{-1}) &= x
\end{align}
\hrule
\caption{Md:  axioms for a meadow}
\label{Md} 
\end{table}
\section{Meadows and vector space meadows}
\label{MVSM}
Numbers will be viewed as elements of a meadow rather than as elements of a field. 
For the introduction of meadows and elementary theory about meadows
I refer to~\cite{BT2007,BBP2013,BBP2015} and the papers cited there.
 I will copy the tables of equational 
axioms for meadows and for the sign function which plays a central role below. 

Below $\reals$ will denote some specific choice of a structure for real numbers. The domain $|\reals|$ of $\reals$ is a particular
set, the choice of which depends on how one prefers to define real numbers. There is no preference for a specific choice.
$\reals_0$ is the unique expansion of $\reals$ with a zero-totalised inverse function $\_^{-1}$. $\reals_0$ is referred to 
as the meadow of reals.

With $(\reals_0,\sg)$ the expansion of the meadow $\reals_0$ with the sign function is denoted. 
The following completeness result was obtained in~\cite{BBP2015}.
\begin{theorem} 
\label{Compl1}
A conditional equation in the signature of signed meadows is valid in 
$(\reals_0,\sg)$ if and only if it is provable from the axiom system $\Md+\DO + \Sign$.
\end{theorem}

\begin{table}
\centering
\hrule
\begin{align}
x^2 &= x \cdot x \\
x/y &= x \cdot y^{-1}\\
1(x) &= x/x \\
0(x) &= 1 -x/x \\
x\lhd y \rhd z &= 1(y) \cdot x + 0(y) \cdot z
\end{align}
\hrule
\caption{DO:  axioms for derived operators}
\label{DO} 
\end{table}

The  axioms in Table~\ref{Md} specify the variety of meadows, while Table~\ref{DO} introduces some  
function symbols by means of defining equations serving as explicit definitions for derived operations. 
Table~\ref{t:a} specifies the sign function, and Table~\ref{Abs} introduces the absolute value function. 
Following~\cite{BBP2013}, a meadow that satisfies the (nonequational) 
implication IL from Table~\ref{IL} is called a cancellation meadow.

\begin{table}
\centering
\hrule
\begin{align}
\label{eq:17}
\sg(1(x))&=1(x)\\
\label{eq:18}
\sg(0(x))&=0(x)\\
\label{eq:19}
\sg(-1)&=-1\\
\label{eq:20}
\sg(x^{-1})&=\sg(x)\\
\label{eq:21}
\sg(x\cdot y)&=\sg(x)\cdot \sg(y)\\
\label{eq:22}
0(\sg(x)-\sg(y))\cdot \sg(x+y)&=0(\sg(x)-\sg(y))\cdot \sg(x)
\end{align}
\hrule
\caption{Sign: axioms for the sign operator}
\label{t:a}
\end{table}

\begin{table}
\centering
\hrule
\begin{align}
|x| &= \sg(x) \cdot x 
\end{align}
\hrule
\caption{ABS: defining axiom the absolute value operator}
\label{Abs}
\end{table}
\begin{table}
\centering
\hrule
\begin{align}
x \neq 0 \to x \cdot x^{-1} = 1 \nonumber
\end{align}
\hrule
\caption{IL:  inverse law}
\label{IL} 
\end{table}

\subsection{Signed vector space meadows}
\label{SVSM}
Let $e_1,\ldots,e_n$ be a series of pairwise distinct objects outside the meadow $\mathbb{M}$, and outside $\Sigma_\Md$. 
The meadow 
$\mathbb{M} \langle e_1,\ldots,e_n\rangle$ is defined as a direct sum of copies of $\mathbb{M}$:
\[\mathbb{M} \langle e_1,\ldots,e_n \rangle = e_1 \cdot \mathbb{M} \oplus\ldots \oplus e_n \cdot \mathbb{M}\]
Here the $e_i$ serve as new constants for orthogonal ($e_i \cdot e_j = 0$ for $i \neq j$) idempotents ($e_i \cdot e_i = e_i$)
such that the set $\{e_1,\ldots,e_n\}$ is complete ($e_1 + \ldots + e_n =1$). Moreover it is assumed that $\sg(e_i) = e_i$. 
Elements of this structure are given by sequences
$(l_1,\ldots,l_n) \in \mathbb{M}^n$ representing the object $e_1 \cdot l_1 + \ldots e_n \cdot l_n$. The meadow operations and sign
are performed coordinate-wise, e.g. $\sg(e_1 \cdot l_1 + \ldots + e_n \cdot l_n) = e_1 \cdot \sg(l_1) + \ldots + e_n \cdot \sg(l_n)$, 
thus obtaining an $n$-dimensional vector space over $\mathbb{M}$.
For $n=1$ the construction brings noting new:  $\mathbb{M}\langle e_1\rangle \cong \mathbb{M}$. For $n>1$, and assuming that $\mathbb{M}$ is non-trivial ($\mathbb{M} \models 0 \neq 1$) the resulting structures are not cancellation meadows, i.e. 
$\mathbb{M} \langle e_1,\ldots,e_n \rangle  \not \models \mathrm{IL}$. 

$\Sigma_{\Md, e_1,\ldots,e_n}$ is the signature $\Sigma_\Md$ expanded with constants $ e_1,\ldots,e_n$.
$\mathbb{M}_{e_1,\ldots,e_n} \langle e_1,\ldots,e_n \rangle$ is the expansion of 
$\mathbb{M} \langle e_1,\ldots,e_n \rangle $ with (the $e_i$ serving as names for  the new orthogonal idempotents. 
Now 
$\mathbb{R}_0(\sg)_{e_1,\ldots,e_n} \langle e_1,\ldots,e_n \rangle \models 
	\Md + Sign + E_{\langle e_1,\ldots,e_n \rangle}$, where $E_{ \langle e_1,\ldots,e_n \rangle}$ captures
the mentioned identities involving the $e_i$: idempotence for the $e_i$, orthogonality for $e_i$ and $e_j$ with $i\neq j$, 
 completeness, and  the equations for $\sg(-)$.

\begin{problem} Is the  axiom system $\Md + Sign + E_{\langle e_1,\ldots,e_n\rangle }$ 
complete for the equational theory of the structure $\mathbb{R}_0(\sg)_{e_1,\ldots,e_n} \langle e_1,\ldots,e_n \rangle $?
\end{problem}

%
With disjunctive assertions (among which IL $\equiv  1(x) = 0 \vee 1(x) = 1$) discrimination between 
vector space meadows of different dimension is possible.\\ Let 
$\phi \equiv_{def} x \cdot x = x \wedge y \cdot y = y \wedge x+ y = 1 \wedge x \cdot y = 0 \to (x = 0 \vee y = 0)$.
Then $\mathbb{R}_0(s)\langle e_1 ,e_{2} \rangle  \not \models \phi$ while : $\mathbb{R}_0(s)\langle \rangle   \models \phi$.

\subsection{Representing functions by expressions}
\label{Conventions}
The expression language may be extended with lambda abstraction thereby introducing $\lambda x.t$ as an
expression denoting the function which maps $v \in V$ to $[v/x]t$, i.e. the result of substituting $v$ for $x$ in $t$. A disadvantage of this
approach is that it imports typed $\lambda$-calculus, definitely a non-trivial subject.

Another option is to use $L\, x.t$ to represent the same function. Now if $y$ does not occur freely in $t$, then 
$L\, y.[y/x]t$
constitutes a different representation for the same function, i.e. unlike in the $\lambda$-calculus alpha conversion does not apply
to $L\, x.t$.

In statistical theory Jeffrey's notation $t[\bullet]$, with $t[-]$ a context with zero or more ``holes'', 
stands for $\lambda x.t[x]$, with $x$ a fresh variable. Finally function abstraction may be left implicit when a specific binding mechanism is employed.

When summation over a bound variable $x$ is applied to a term $t$ or to a context $t[-]$, these four options lead to different notations: 
$\sum^\star(\lambda x.t),\sum^\star (L\,x.t),  \sum^\star t[\bullet]$ and $\sum^\star_x t$, respectively. 
There is no need to choose a single convention from 
these four options and below it is supposed  to be clear from the 
context which one of these conventions is used in each particular case.

\subsection{Finite support summation}
\label{FSS}
Given a meadow $\mead$ and a term $t$ in which variable $x$ may or may not occur 
it may be useful to determine the summation of 
all substitutions (or rather interpretations) $[v/x]t$ with $v \in | \mead |$. 
This sum is unambiguously defined, however, if the support in $\mead$ of $L\, x.t$ is finite, that is if there are only 
finitely many values $v \in | \mead |$ such that $[v/x]t$ is nonzero.

The expression $\sum_{x}^\star t $ denotes in $\mead$ the sum of all $[v/x]t$ if at most finitely 
many of these substitutions $[v/x]t$ yield a non-zero value and 0, otherwise. 

The $\sum_{x}^\star $ operator will be referred to as finite support summation (FSS). At this stage we have little information about the logical properties of this binding mechanism on terms but it is semantically unproblematic, being
well-defined in each meadow,  and it will be used below for presenting several  definitions. 
We first notice some technical facts concerning FSS, assuming the interpretation of equations is performed in  
an arbitrary  cancellation meadow $\mead$. 
\begin{enumerate}
\item $L\, x.t$ has finite support iff $L\, x.t/t$ has finite support.
\item $\sum_{x}^\star 0=0$, $\sum_{x}^\star 0(x) =1$, 

\item $\sum_{x}^\star 1 = 0$. To see this first notice that in an infinite  meadow 1 is nonzero for infinitely many $x$ and thus $\sum_{x}^\star 1 = 0$. A finite meadow has nonzero characteristic (say $p$) and $\sum_{x}^\star 1$ counts up to the cardinality of the structure, which is a multiple of $p$ and therefore vanishes modulo $p$.  
\item  $\sum_{x}^\star 1(x)= 0$ if and only if $\mead$ is infinite.
\item  $\sum_{x}^\star 1(x )= -1$ if and only if $\mead$ is finite.
\item $\sum_{x}^\star (t +0(x)) = (\sum_{x}^\star t) + 1$ if and only if $L\, x.t$ has finite support. 
\item $\sum_{x}^\star (t +0(x)) = \sum_{x}^\star t$ if and only if $L\, x.t$ has infinite support.

\item If $x \notin FV(t)$ then $\sum_x^\star(r \cdot t) = (\sum_x^\star r)  \cdot t)$.
\item If $x \notin FV(t)$ then $\sum_x^\star(x \cdot 0(t-x)) =  t$ and $\sum_x^\star(x \cdot 1(t-x)) =  (\sum_x^\star t) - [0/x]t$.
\item If both $L\, x.t$ and $L\, x.t$ have finite support then $(\sum_x^\star t)+(\sum_x^\star r) = \sum_x^\star (t+r).$
\end{enumerate}
If, moreover, $\mead$ is signed:
\begin{enumerate}
\item $\sum_{x}^\star 1(x) =0$, because a signed meadow is infinite.
\item Consider context $ C[-]$ with  $C[X] = 1(\sum_{x}^\star 1(X))\lhd (\sum_{x}^\star (X +0(x)) - \sum_{x}^\star X)\rhd 1,$ 
then $C[t] = 1$ if and only if the support of $L\, x.t$ is nonempty, and otherwise $C[t] = 0$.
\item $C[t] \cdot (\sum_{x}^\star (t +0(x)) - (\sum_{x}^\star t)) \cdot 0(1-\sum_{x}^\star \sg(t ))=0$ if  and only if the support of 
$L\, x.t$ is a singleton.
\end{enumerate}
\begin{proposition} 
\label{SuppNe}
$L\, x. t$ has finite support in $\mathbb{Q}_0$ if and only if it has finite support in $\mathbb{R}_0$,  
\end{proposition}

\begin{proof} Because $\mathbb{Q}_0$ is  a substructure of $\mathbb{Q}_0$ the number of non-zero values of $\lambda x.t$ in 
$\mathbb{Q}_0$ cannot exceed the number of nonzero values in  $\mathbb{R}_0$ so the if part is immediate. 
Now for ``only if'' suppose that $\lambda x.t$
has infinitely many non-zero values in $\mathbb{R}_0$. In~\cite{BBP2013} it is shown that non-zero $t(x)$ is provably equal to a sum of
simple fractions, i.e. fractions for which numerator and denominator are each nonzero-polynomials.
This implies that $\lambda.t(x)$ is discontinuous on at most finitely many arguments so that it must be nonzero at some real argument $r$ where it is continuous at the same time. This implies that $\lambda.t(x)$ is nonzero in some neighbourhood $(r-\epsilon,r+\epsilon)$ of $r$ so that it is nonzero on the
infinitely many rational arguments in this same neighbourhood. It follows that $\lambda x.t$ has infinite support in $\rats_0$.
\end{proof}
\begin{problem} Is there a context $C[-]$ (not involving $\sg$) so that for all meadow expressions without sign and for all cancellation meadows
(in particular those with non-zero characteristic) $C[t]=0$ equals $0$ if $t$ has empty support and $C[t]=1$ otherwise?

\end{problem}
\begin{problem} Consider the meadows $\reals_0$ enriched with FSS. Is equality between closed terms for this structure 
computably enumerable, and if so is it decidable?
\end{problem}
\begin{problem} Consider the meadows $\rats_0$ enriched with FSS. Is equality between closed terms for this structure 
decidable?
\end{problem}

\subsection{Multivariate finite support summation}
The multivariate case of FSS operations requires separate definitions for each 
number of variables because a stepwise reduction to  the definition for the univariate case is unfeasible. 
To demonstrate this difficulty we consider the bivariate case only, the case with three or 
more variables following the same pattern. In a meadow $\mead$, 
$\sum_{x,y}^\star t $ produces 0 if for infinitely may pairs of values $a,b \in  | \mead |$ 
the value of $[a/x][b/y]t$ is nonzero, otherwise 
it produces the sum of the finitely many nonzero values thus obtained. 

The need for expressions of the form $\sum_{x,y}^\star t $ transpires from an elementary example, 
which demonstrates that a 2-dimensional FSS cannot be simply reduced to a composition of 2 occurrences of a 
1-dimensional FFS.
Let $t(x,y) = 0(x) \cdot  0(y) + 0(1-x).$ 
Because  $t(1,y)= 1$ for all $y$, $t(x,y)$ is nonzero on infinitely many pairs of values,  so that
\[\sum_{x,y}^\star  t(x,y)=0.\]
Now notice that $\sum_{y}^\star  t(0,y)=1$,  $\sum_{y}^\star  t(1,y)=0$, and if $x\neq 0 \wedge x \neq 1$, 
$\sum_{y}^\star  t(x,y)=0$. It follows that
\[\sum_{x}^\star \sum_{y}^\star  t(x,y)  =1.\]

\section{Probability mass functions  with finite support}
\label{PMFswfs}

The main application of FSS in this paper is to enable the following definition of what it means for a term to represent a 
finitely supported probability mass function. Probability mass function will be abbreviated as PMF.
Finitely supported PMFs constitute a special case of ``arbitrary'' PMFs , a more general notion which cannot easily be defined on an 
arbitrary signed meadow, and which will not be used in the sequel.
\begin{definition}
\label{PMFwfs}
Given a signed meadow $\mead$, a pair $(t;x)$ consisting of a term and variable $x$, represents a 
PMF with finite support  $\mead$ if $\mead \models |t| = t$ and $\mead \models \sum_{x}^\star t = 1$.

\end{definition}
A PMF with finite support is also called a finitary PMF or a finitely supported PMF. The use of terminology from probability theory 
requires some justification. Indeed PMFs occur in probability theory where these comprise precisely all
nonnegative functions from reals to reals with a countable support so that the sum of all non-zero values equals 1. 
With this fact in mind, and working in the signed meadow $\reals_0(\sg)$, the two requirements of 
Definition~\ref{PMFwfs} indeed guarantee that the function represented by $L\,x.t$ is a PMF 
with finite support according to standard terminology.

The property of being a representative of a finitely supported PMF is sensitive to the meadow at hand.
For instance consider the expression $t$ given by
\[ t = 0(x^2-2)\cdot ((1+ \sg(x))\cdot  x +(1- \sg(x)) \cdot (2-x))/4.\] In $\reals_0$ the function description $L\,x.t$ represents a finitary PMF. To see this notice that $L\,x.t$  takes non-zero values only in $-\sqrt{2}$ and $\sqrt{2}$  where it has values $1 -\sfrac{1}{2} \cdot \sqrt{2}$ and 
$\sfrac{1}{2}
\cdot  \sqrt{2}$ respectively, so that $L\, x.t$ represents a finitary PMF, while  in $\rats_0$ it is not the case that $L \,x.t$ represents a 
finitary PMF  because $t(q)$ vanishes for all $q \in \rats_0$ with the implication that
$\sum^\star_x t = 0$.
On the other hand when considering $t^\prime(x) = t(x)+ 0(x)$  it turns out that $L\, x.t$ represents a finitely supported 
PMF in $\rats_0$ while it fails to do so in $\reals_0$.

\subsection{Multivariate PMFs with finite support}
Given a signed cancellation meadow $\mead$, a joint PMF with finite support of arity $n$ is a 
function $L\,x_1,\ldots, x_n.F(x_1,\ldots,x_n)$ 
from $\mead^n$ to $\mead$
which satisfies these two conditions: 
\begin{enumerate}
\item $\sum_{x_1,\ldots,x_n}^\star F(x_1,\ldots,x_n) = 1,$ and
\item for all $x_1,\ldots, x_n \in \reals_0^n$, $F(x_1,\ldots,x_n) = |F(x_1,\ldots,x_n)|.$ 
\end{enumerate}

For example assuming that information about the graph of a joint 
PMF with finite support, with exception of argument vectors for which the result vanishes,  is
encoded in a set of triples:
 $\{(y_{1,1},y_{2,1},z_1),\dots,(y_{1,n},y_{2,n},z_n)\},$ a corresponding function expression $F$ for the same joint 
 PMF with key variables $x_1$ and $x_2$  is as follows:
\[F(x_1,x_2) = 
\sum_{i=1}^{n} ( 0(x_{1}-y_{1,i}) \cdot 0(x_{2}-y_{2,i}) \cdot z_i)\]

\subsection{Marginalisation and independence}
Given a finitely supported joint PMF $G$ with $n$ variables $x_1,\ldots,x_n$, 
marginalisation can be defined to each subset $x_{i_1},\ldots,x_{i_k}$ with $1 \leq i_1 < \ldots < i_k \leq n$. Let $x_{j_1},\ldots,x_{j_{n-k}}$ be an enumeration without repetition 
of the variables in $x_{1},\ldots,x_{n}$ that are not listed in $x_{i_1},\ldots,x_{i_k}$, then  
$G_{(i_1,\ldots,i_{k})}$ represents a joint PMF with $k$ variables $x_{i_1},\ldots,x_{i_k}$ as follows:
\[G_{(i_1,\ldots,i_{k})}(x_{i_1},\ldots,x_{i_k}) = \sum_{x_{j_1}\ldots,x_{j_{n-k}}}^\star G(x_1,\ldots,x_n) \]

For a bivariate PMF $G(x,y)$ independence is defined as independence of its two marginalisations. 
\[\mathit{IND}(G) \equiv_{def} \forall x,y \in V. G(x,y) = G_{(1)}(x) \cdot G_{(2)}(y).\]

\subsection{Expectation, (co)variance, and correlation for PMFs}
Now $F(x)$ is assumed to be a term representing a finite support PMF with $x$ as the key variable, while
$G(x,y)$ represents a joint PMF with finite support with $x$ as the first and $y$ as the second key variable. 
Two PMFs $G_{(1)}$ and $G_{(2)}$ are derived from $G$ by marginalization: 
$G_{(1)}(x) = \sum_{y}^{\star}G(x,y)$ and $G_{(2)}(y) = \sum_{x}^{\star}G(x,y)$.
The expected value $E_{\mathit{pmf}}(F)$ of  $F$ and related operations are given in Table~\ref{EVCCpmf}.
\begin{table}
\centering
\hrule
\begin{align}
E_{\mathit{pmf}}(F) &= \sum_{x}^{\star}(x \cdot F(x)) \tag{Expected value of F}\\
\mathit{VAR}_{\mathit{pmf}}(F) &= \sum_{x}^{\star}(x^2 \cdot F(x)) - E_{\mathit{pmf}}(F)^2 \tag{variance of F}\\
\mathit{COV}_{\!\mathit{pmf}}(G) &= \sum_{x,y}^{\star}(x \cdot y \cdot G(x,y)) - E_{\mathit{pmf}}(G_{(1)}) \cdot E_{\mathit{pmf}}(G_{(2)})
	\tag{covariance of G}\\
\mathit{CORR}_{\mathit{pmf}}^{sq}(G) &= 
\frac{\mathit{COV}_{\!\mathit{pmf}}(G)^2 }{ (\mathit{VAR}_{\mathit{pmf}}(G_{(1)}) \cdot \mathit{VAR}_{\mathit{pmf}}(G_{(2)})}
	\tag{correlation of G squared}
\end{align}
\hrule
\caption{expected value, (co)variance, and correlation}
\label{EVCCpmf} 
\end{table}
The square of correlation is included in order not to burden the present exposition
with the equational specification of a square root operator. In the context of meadows the square root can be made total,
and the equationally specified, by writing 
$\sqrt(-x) = -\sqrt(x)$ (see~\cite{BBP2013}).
The completeness result of Theorem~\ref{Compl1} carries over in the presence of the square root function. 

These definitions admit a justification  on the basis of the conventional use of the defined terminology, 
the details which  are worth mentioning.
Given a PMF $F$ with finite support, its support, say $S$, may be viewed as a 
sample space so that in conventional terminology $id_S$, the identity function of type $S \to \reals_0$, 
qualifies as a random variable,
say $X$. The power set of $S$ serves as an event space, say $\events_S$. 
Let the probability function $P$ be generated by  
$P(\{s\}) = F(s)$  for $s \in S$. Now $P(X = x) = F(x)$ and
$E_{\mathit{pmf}}(F) = E_P(X) = \sum_{s \in S} (X(s) \cdot P(X=s) )= \sum_{x}^{\star}(x \cdot F(x))$.

\section{Event spaces and probability functions}
\label{PFsCVs}
From~\cite{BP2016} I will recall  equations for Boolean algebras, (signed) meadows, 
and probability functions.
A Boolean algebra $(B,+,-,\overline{~},1,0)$ may be defined as a system with at 
least two elements such that $\forall x,y,z\in B$ the 
well-known postulates of Boolean algebra are valid.
In order to avoid overlap with the operations of a meadow, 
Boolean algebras are equipped with notation from propositional logic,
thus consider $(B, \vee,\wedge,\neg,\top,\bot)$ and adopt the axioms as presented in 
Table~\ref{Ba}.
In~\cite{Pad83} it was shown that the axioms in Table~\ref{Ba} constitute an equational basis for the 
equational theory of Boolean algebras.
\begin{table}
\centering
\hrule
\begin{align}
\label{eq:1}
(x\vee y)\wedge y&= y\\
\label{eq:2}
(x\wedge y)\vee y&= y\\
\label{eq:3}
x\wedge (y\vee z) &= (y\wedge x)\vee (z\wedge x)\\
\label{eq:4}
x\vee (y\wedge z) &= (y\vee x)\wedge (z\vee x)\\
\label{eq:5}
x\wedge \neg x &= \bot\\
\label{eq:6}
x\vee\neg x&=\top
\end{align}
\hrule
\caption{BA: a self-dual equational basis for Boolean algebras}
\label{Ba}
\end{table}
In the setting of probability  functions the elements of the underlying 
Boolean algebra are referred to as events.
  Events are closed under $-\vee-$ which represents alternative occurrence 
  and $-\wedge-$ which represents simultaneous occurrence.
The term ``value'' will refer to an element of a cancellation meadow, mainly the meadow of reals and the meadow of rationals.
A probability  function from events to the values in a signed meadow.
An expression of sort $E$ is an event expression or an event term, an expression of 
type $V$ is a value expression or equivalently a value term.  
In this paper considerations are limited to structures involving a single name for a probability function only, 
the function symbol $P$, at least in the  1-dimensional case.
Table~\ref{PFBC} provides axioms that determine generally agreed boundary conditions for a probability function. 
Table~\ref{PF} contains the axiom for additivity that is included in the axiomatisation of~\cite{BP2016}.
Together with the axioms for signed meadows and for Boolean algebras we find 
the following set of axioms: $\BA+\Md+\DO+\Sign+\ABS+\PFBC_{\!P}+ \PFA_P$.

\begin{table}
\centering
\hrule
\begin{align}
\label{eq:23}
P(\top) &=1 \\
\label{eq:24}
P(\bot) &=0\\
\label{eq:25}
P(x) & = |P(x)|
\end{align}
\hrule
\caption{PFBC$_P$:  boundary conditions for a probability  function }
\label{PFBC}
\end{table}

\begin{table}
\centering
\hrule
\begin{align}
P(x \vee y) &= P(x)+P(y) - P(x \wedge y)  
\end{align}
\hrule
\caption{PFA$_P$:  addtivity axiom for a named probability  function }
\label{PF}
\end{table}

Table~\ref{CPOs} provides explicit definitions of some useful 
conditional probability operators made total by choosing a value in
case the condition has probability 0.

\begin{table}
\centering
\hrule
\begin{align}
P^0(x\midd y) &= \frac{P(x \wedge y)}{P(y)} \\
P^1(x\midd y) &= P^0(x\midd y) \lhd P(y) \rhd 1 \\
P^s(x\midd y) &= P^0(x\midd y) \lhd P(y) \rhd P(x) 
\end{align}
\hrule
\caption{conditional probability operators}
\label{CPOs}
\end{table}

\subsection{Soundness and completeness of axioms for probability functions}
The reader is assumed to be familiar with the concept of  a probability function, say $\pf$ with name $P$, on an event 
space $\events$, where $\pf$ is supposed to comply with the informal Kolmogorov axioms of probability theory.\footnote{%
I use $\events$ for a specific Boolean algebra/event space in order to indicate that it is a parameter 
on which other parameters such as $\pf$ depend. A difference between $P$ and $\pf$ is made because for one name $P$ different interpretations are considered. Assimilation of $P$ and $\pf$, like with the constant $0$ and the real number $0$ is only 
adequate, in some context, if, relative to the context at hand, 
commitment to the use of a single probability function has been made.}
Being based on
the availability of real numbers, sets, and measures on sets, the Kolmogorov axioms are more easily 
understood as providing a mathematical definition, that is a set of requirements, governing which functions are considered 
probability functions than as 
constituting a formal system of axioms.
The axiom system $\BA+\Md+\DO+\Sign+\ABS+\PFBC_{\!P} + \PFA_P$ may be considered a formalisation of the 
Kolmogorov axioms for 
probability functions. 

A probability function structure over an event space $\events$ is a two sorted structure having 
$E$ (events) and $V$ (values) as sorts with $E$ interpreted by a Boolean algebra, now denoted $\events$,
 and $V$ interpreted as the 
real numbers $\reals$ as chosen in Section~\ref{MVSM}, enriched with a probability function $\pf$ from $E$ to $V$. 
The Kolmogorov axioms specify precisely which functions are probability functions. I will assume that $V$ is the domain of the meadow of reals, i.e. that the meadow version of real numbers is used. With $\EPV(\events,\reals_0(\sg),P)$ the class of probability function structures 
over a fixed event structure $\events$ is denoted, with values taken in $|\reals_0(\sg)|$. 
For a specific PMF $\pf$ the pertinent structure is denoted by
 $\EPV(\events,\reals_0(\sg),\pf)$. $\EPV(\BA,\reals_0(\sg),P)$ denotes the union of all collections 
 $\EPV(\events,\reals_0(\sg),P)$ for all $\events$ with $\events \models \BA$.
It is apparent from the construction  that $\EPV(\events,\reals_0(\sg),\pf) \models \BA+\Md+\DO+\Sign+\ABS+\PFBC_{\!P} + \PFA_P$.
A completeness result for $\BA+\Md+\DO+\Sign+\ABS+\PFBC_{\!P} + \PFA_P$ is taken from~\cite{BP2016}.
\begin{theorem} 
\label{Compl:1}
$\BA+\Md+\DO+\Sign+\ABS+\PFBC_{\!P} + \PFA_P$ is sound and complete for the  
equational theory of $\EPV(\BA,\reals_0(\sg),P)$.
\end{theorem}
It is a corollary of the completeness proof in~\cite{BP2016} that the same axioms are complete for the 
class $\EPV(\BA^f,\reals_0(\sg),P)$
containing those probability function structures  which are expansions of a finite event structure.
In~\cite{Halpern1990} first order axioms are provided for probability calculus, and corresponding completeness is shown 
making use of the completeness result for the first order theory of real numbers, a fact which also underlies the result in~\cite{BP2016}.

\subsection{$\mathsf{BR }$ and $\mathsf{BRs}$, two forms of Bayes' rule}
As a comment to the specification of probability functions an excursion to Bayes' rule is worthwhile.   
First consider the following equation:
\begin{equation} 
	P(x \wedge y)\cdot P(y)\cdot P(y)^{-1} = P(x \wedge y) \tag{$\mathsf{EQ1}$}
\end{equation}
 Equation~$\mathsf{EQ1}$ follows from 
$\BA+\Md+\DO+\Sign+\ABS+\PFBC_{\!P} + \PFA_P$. 
This fact is a consequence of Theorem~\ref{Compl:1} above. A direct proof reads as follows.

$\displaystyle \phi(u,v) \equiv 0(|u| + |v|) \cdot u.$ Now $(\reals_0,\sg)\models \phi(u,v) = 0$, and
using the completeness theorem of~\cite{BBP2015} one obtains that $\BA+\Md+Sign \vdash \phi(u,v) = 0$. Substituting 
$P(y \wedge x)$ for $u$ and $P(y \wedge \neg x)$ for $v$ one derives: 
	$\displaystyle  \vdash 0 = \phi(P(y \wedge x),P(y \wedge \neg x)) =
	0(|P(y \wedge x)| + |P(y \wedge \neg x)|) \cdot P(y \wedge x) =
	0(P(y \wedge x) + P(y \wedge \neg x)) \cdot P(y \wedge x) =
	 0(P(y)) \cdot P(y \wedge x)$,
	from which the required result follows by expanding $0(P(y))$.

Bayes' rule, also known as Bayes' theorem,  occurs in different forms. 
The conditional operator $P^0$ of Table~\ref{CPOs} is used for its presentation below. 
The simplest form of Bayes' rule, is an equation here referred to as $\mathsf{BR }$:
\begin{equation} P^0(x \midd y) = \frac{P^0(y \midd x) \cdot P(x)}{P(y)} \tag{$\mathsf{BR}$}\end{equation}
In~\cite{BP2016} it is shown that $\mathsf{BR}$ follows from the specification $\BA+\Md+\DO+\Sign+\ABS+\PFBC_{\!P}+\mathsf{EQ1}$. 
 As it turns out $\mathsf{BR}$ implies equation~$\mathsf{EQ1}$.
This fact  is shown as follows: by substituting $x \wedge y$ for $y$ in $\mathsf{BR}$
one obtains: $P^0(x|x \wedge y) = (P^0(x \wedge y|x) \cdot P(x))/P(x \wedge y)$. 
Multiplying both sides with $P(x \wedge y)$ gives $L = R$ with $ L = P^0(x|x \wedge y)\cdot  P(x \wedge y)$ and 
$R= ((P^0(x \wedge y|x) \cdot P(x))/P(x \wedge y))\cdot P(x \wedge y)$. Now 
$L = (P(x \wedge (x \wedge y))/P(x \wedge y)) \cdot P(x \wedge y) = (P(x \wedge y) \cdot P(x \wedge y)) / P(x \wedge y) = P(x \wedge y)$, and $R = (((P((x \wedge y) \wedge x)/P(x)) \cdot P(x))/P(x \wedge y))\cdot P(x \wedge y) = 
(P(x \wedge y)/P(x \wedge y))\cdot P(x \wedge y) \cdot (P(x)/ P(x)) = P(x \wedge y) \cdot P(x) \cdot P(x)^{-1}.$

\begin{proposition} The axiom system $\BA+\Md+\DO+\Sign+\ABS+\PFBC_{\!P} +\mathsf{EQ1}$ is
strictly weaker than $\BA+\Md+\DO+\ABS+\Sign+\PFBC_{\!P} + \PFA_P$. 
\end{proposition}
\begin{proof}
Consider a four element 
event space generated by an atomic event 
$e$ and choose $\pf$ as follows: $\pf(\bot) = \pf(e) = \pf(\neg e) = 0$ and $\pf(\top) = 1$. The equations of 
$\PFBC_{\!P}$ and $\mathsf{EQ1}$ are satisfied  while $\PFA_P$ is not satisfied.
\end{proof}
This weakness persists if~$\mathsf{EQ1}$is replaced by $\mathsf{BR}$.
A second and equally well-known form of Bayes' rule is $\mathsf{BRs}$ from Table~\ref{PFA2}.
$\mathsf{BR}$ follows from $\BA+\PFBC_{\!P}+ \mathsf{BRs}$ by taking $z = \top$. 

\begin{proposition}
$\BA+\PFBC_{\!P}+\mathsf{BRs}$  implies $\PFA_P$.
\end{proposition}
\begin{proof}  It suffices to derive the following equation~\ref{EQ2}
	\begin{equation} 
	\label{EQ2} 
	P(y) = P(y \wedge z) + P(y \wedge \neg z) \tag{$\mathsf{EQ2}$}
	\end{equation}
This suffices because, according to~\cite{BP2016}, it is the case that $\mathsf{EQ2}$ in combination with $\BA+\Md+\DO+\Sign+\ABS+\PFBC_{\!P}$ entails $PFA_P$.
To this end set $x= y$ in $\mathsf{BR}_2$, thereby obtaining $P^0(y |y) = (P^0(y|y) \cdot P(y))/(P^0(y|z)\cdot P(z) + P^0(y| \neg z)\cdot P(\neg z)).$ 

 To derive equation ~$\mathsf{EQ2}$, notice $P^0(y |y) = P(y \wedge y)/P(y) = P(y)/P(y)$,  take the inverse at both sides thus obtaining 
 $L = R$ with $L=P(y)/P(y)$ and $R= (P^0(y|z)\cdot P(z) + P^0(y| \neg z)\cdot P(\neg z))/P(y)$. 
 Then multiplying $L$ and $R$ with $P(y)$ yields $L \cdot P(y) = R \cdot P(y)$. Now
 $L \cdot P(y) = (P(y)/P(y)) \cdot P(y) = P(y)$ and 
 $R \cdot P(y) = ((P^0(y|z)\cdot P(z) + P^0(y| \neg z)\cdot P(\neg z))/P(y))\cdot P(y)=\\
 ((P(y \wedge z)/P(z))\cdot P(z) + (P(y \wedge \neg z)/P(\neg z))\cdot P(\neg z))\cdot (P(y)/P(y))=\\
  (P(y \wedge z) + P(y \wedge \neg z)) \cdot (P(y)/P(y)) =
 P(y \wedge z) \cdot (P(y)/P(y)) + P(y \wedge \neg z) \cdot (P(y)/P(y)) =\\
  P(y \wedge z) + P(y \wedge \neg z)$.
\end{proof}
  It may be concluded that Table~\ref{PFA2} provides an adequate substitute of $\PFA_P$. 
   This observation suggests an alternative axiomatisation $\BA+\Md+\DO+\Sign+\ABS+\PFBC_{\!P}+\PFA^\prime_P$ 
   based on $\mathsf{BRs}$ as given in Table~\ref{PFA2}.
  
  For $\mathsf{BR}$, however, there seems to be no role as an axiom in the axiomatic framework of this paper. 
  For instance one may wonder if $\mathsf{BR}$
  provides an implicit definition of conditional probability.
\begin{proposition} It is not the case that in the presence of $\BA+\Md+\DO+\Sign+\ABS+\PFBC_{\!P}+\PFA_{P}$, 
though in the absence of the definitions of Table~\ref{CPOs},  $\mathsf{BR}$ serves as an implicit definition of $P^0$.
\end{proposition}
\begin{proof} Let $\displaystyle Q(x,y) =1(P(y)) \cdot P(x)$. Then $Q(x,y)$ differs from $P^0(x\midd y)$ in all but exceptional cases. 
However, $Q(-,-)$ satisfies $\mathsf{BR}$ considered as  a requirement on $P^0(- \midd -)$: $\displaystyle \frac{Q(y,x) \cdot P(x)}{P(y)}= 
\frac{1(P(x)) \cdot P(y)\cdot P(x)}{P(y)}=1(P(y)) \cdot 1(P(x)) \cdot  P(x) =1(P(y)) \cdot P(x) = Q(x,y)$. 
\end{proof}
 
 \begin{table}
\centering
\hrule
\begin{align}
P^0(x\midd y) &= \frac{P^0(y\midd x) \cdot P(x)}{P^0(y\midd z)\cdot P(z) + P^0(y\midd  \neg z)\cdot P(\neg z)} \tag{$\mathsf{BRs}$}
\end{align}
\hrule
\caption{PFA$^\prime_P$:  alternative axiom for additivity}
\label{PFA2}
\end{table}

\begin{table}
\centering
\hrule
\begin{align}
	v(- x) &= -v(x) \\
	v(x^{-1}) &= v(x)^{-1} \\
	v(x+y) &= v(x) + v(y)\\
	v(x \cdot y) &= v(x) \cdot v(y)\\
	v(\sg(x)) &= \sg(v(x))
\end{align}
\hrule
\caption{$\UCV$:  axioms for unconditional values; $x,y$ range over $V$.}
\label{UCV} 
\end{table}

 \subsection{A signed vector space meadow of conditional values}
 \label{MofCV}
A third sort named $V_C$ containing so-called conditional values will be introduced. 
$V_C$ is generated by an embedding $v\colon V \to V_C$ and
a conditional operator $ -\! : \to -\colon E \times V_C \to V_C$. $V_C$ is equipped with all meadow operations while $v(0)$ 
serves as 0 and $v(1)$  serves as 1. A specification is given by combining (i) the axioms
$\UCV$ of Table~\ref{UCV} with (ii) $\Md_{cv} =\Md_{[v(0)/0,v(1)/1]}$, i.e the equations of Table~\ref{Md}, 
however with variables $X,Y,Z$ now ranging over $V_C$, 
and with $v(0)$ substituted for $0$ and $v(1)$ substituted for $1$, (iii) $\Sign_{cv}$, the equations of Table~\ref{t:a}, but
now with its variables ranging over $V_C$ and writing $v(0)$ for $0$ and $v(1)$ for $1$,
and (iv) the specification $\Cond$
of the conditional operator  $ -\! : \to -\colon E \times V_C \to V_C$ as specified in Table~\ref{CondEqs}.

Given a Boolean algebra $\events$ and a signed meadow $\mead(\sg)$ there is a three sorted algebra 
$\UCV(\events,\mead(\sg),\cv(\events,\mead(\sg)))$
with  the domain $\cv(\events,\mead(\sg))$ for sort $V_C$  freely generated from $\events$ and $\mead(\sg)$, of 
which  includes a sort $V_C$, the conditional operator on $E \times V_C$, 
 and the embedding $v$  from $V$ into $V_C$. 

For a Boolean algebra $\events$ the subset $\events_{\mathit{at}}$  consists of the atomic elements of $|\events|$,
where $a \in |\events|$
 is atomic if $a \neq \bot$ and whenever for $b$ and $c$ in $|\events|$, 
 $\events \models (\neg b \vee a) \wedge (\neg c \vee a) =\top $ then  $\events \models \neg b \vee a = \top$ or 
 $\events \models \wedge \neg c \vee a = \top$. $\events_{\mathit{at}}$ contains the maximally consistent elements of the Boolean algebra.

\begin{table}
\centering
\hrule
\begin{align}
\top:\to X &=X\\
\bot:\to X &=v(0)\\
 e\!: \to (X+Y) &= (e\!: \to X) + (e\!: \to Y) \\
 \label{dist:mult}
e\!: \to (X \cdot Y) &= (e\!: \to X) \cdot Y \\
e\!: \to (-X) &= -(e\!: \to X)  \\
e\!: \to (X^{-1}) &= (e\!: \to X)^{-1}  \\
(e \vee f\!: \to X)   &=( e\! : \to X) +  (f\!: \to X) - (e \wedge f\!: \to X)\\
e \wedge f\!: \to X &= e \!: \to( f\!: \to X)\\
\sg(e\!: \to X) &=e\!: \to \sg(X)
\end{align}
\hrule
\caption{$\Cond$: axioms for the conditional operator}
\label{CondEqs}
\end{table}
 
To each closed term $X$ of type $V_C$ of the extended 
signature a mapping $\llbracket X \rrbracket: \events_{\mathit{at}} \to V$ is assigned, with the rules of  Table~\ref{FunCV}.
The equivalence relation $\equiv_{at}$ on closed $V_C$ terms is given by 
$X \equiv_{at} Y \iff \forall a \in \events_{\mathit{at}} (\llbracket X \rrbracket(a) = \llbracket Y \rrbracket(a))$.
$\equiv_{at}$ is a congruence relation  which meets all requirements imposed by 
$\UCV+ \Sign_{cv} + \Md_{cv}+ \Cond$
and $\cv(\events,\mead)$  can be defined as the free term algebra 
for sort $ V_C$ in the extended signature modulo $\equiv_{at}$. This construction guarantees the consistency of the given 
construction of 
the structure for $V_C$
as for arbitrary $a \in \events_{\mathit{at}}$: $\llbracket v(0) \rrbracket(a) = 0 \neq 1 =  \llbracket v(1) \rrbracket(a)$.

\begin{proposition} 
\label{Noncancellation}
If $\mead$ is nontrivial (that is $\mead \not\models 0 = 1$) and $|\events|$ has 
more than two elements then $\cv(\events,\mead)$ is not a cancellation meadow 
(that is $\cv(\events,\mead) \not\models X \neq 0 \to X \cdot X^{-1} = 1$).
\end{proposition}
\begin{proof}  
The proof works by finding an $X$ which differs from $v(0)$ modulo $\equiv_{at}$
and so that $X \cdot X^{-1}$ differs from $v(1)$ modulo $\equiv_{at}$. 
Indeed If $|\events| > 2$ then $\events_{\mathit{at}}$ is non-empty, and let $a$ be an atom. 
Now $ a\!:\to v(1)$,  violates IL. First notice that $\llbracket \bot\!:\to v(1) \rrbracket (a) = 0 \neq 
1 = \llbracket  a\!:\to v(1) \rrbracket (a)$ so that $\bot\!:\to v(1) \not\equiv_{at} a\!:\to v(1)$, and similarly
by application to $\neg a$ that 
$a\!:\to v(1) \not\equiv_{at} \top \!:\to v(1)$. Now $(a\!:\to v(1))^{-1} = a\!:\to v(1)^{-1} = a\!:\to v(1^{-1}) = 
e\!:\to v(1)$ whence
$(a\!:\to v(1)) \cdot (a\!:\to v(1))^{-1} = (a\!:\to v(1)) \cdot (a\!:\to v(1))=
a\!:\to v(1) \not\equiv_{at} \top \!:\to v(1) (\not \equiv_{at} a\!:\to v(1))$.
\end{proof}

\begin{definition}
An expression $X = e_1:\to v(t_1)+\ldots e_n:\to v(t_n)$ of type $V_C$ is a flat $V_C$ expression.
\end{definition}
\begin{definition}
A flat $V_C$ expression $X = e_1:\to v(t_1)+\ldots e_n:\to v(t_n)$ is non-overlapping   
if for all $1 \leq i,j\leq n$ with $ i \neq j$, it is the case that provably $e_i \wedge e_j = \bot.$
\end{definition}
\begin{definition}
Two non-overlapping  flat $V_C$ expressions are similar if both involve the same collection of conditions, used in the same order.
\end{definition}
\begin{proposition} For each closed $V_C$ expression $X$ there is a non-overlapping flat $V_C$ expression $Y$  such that 
 $\Md+\DO+\Sign+\ABS+\UCV + \Sign_{cv} + \Md_{cv}+\Cond \vdash X=Y$.
\end{proposition}
\begin{proposition} For closed $V_C$ expressions $X$ and $Y$ 
similar non-overlapping flat expressions $X^\prime$ and 
 $Y^\prime$ can be found so that $\Md+\DO+\Sign+\ABS+\UCV + \Sign_{cv} + \Md_{cv}+\Cond \vdash X=X^\prime \, \&\, Y=Y^\prime.$
\end{proposition}
\begin{proposition}
If we fix $\events$ as some finite minimal event space with $\events \models \top \neq \bot$,
then the $V_C$ expressions generated from $\events$ and $\reals_0$ constitute a signed vector meadow meadow
with dimension $\#(\events_{\mathit{at}})$. If $\#(\events_{\mathit{at}}) \geq 2$ then the meadow of 
conditional values is not a cancellation meadow. Instead it is a vector space meadow (see Paragraph~\ref{SVSM}). 
Elements of the form
$e\!:\to 1$ with $e \in \mathbb{E}$, are the idempotent elements of $V_C$. CVs $e\!:\to 1$ and $f\!:\to 1$ are orthogonal
if and only if $e \wedge f = \bot$ in $\mathbb{E}$. If $a_1,\ldots a_n$ enumerates $\events_{\mathit{at}}$ without repetition
 then $V_C \cong \mathbb{R}_0(\sg)\langle a_1,\ldots,a_n \rangle $.
\end{proposition}
\begin{proposition} Given closed $V_C$ expressions in flat form $X=\sum_{i=1}^n e_i\!:\to v(t_i)$ and
 $Y = \sum_{j=1}^m f_j\!:\to v(r_j)$, 
a flat form representation for $X \cdot Y$ is:
$ \sum_{i=1}^n \sum_{j=1}^m(e_i \wedge f_j)\!:\to v(t_i \cdot r_j).$ 
Moreover, if $X$ and $Y$ are non-overlapping then so is the 
given expression for $X \cdot Y$.
\end{proposition}
\begin{table}
\centering
\hrule
\begin{align*}
\llbracket v(m) \rrbracket(a) & = m\\
\llbracket -t \rrbracket(a) & = -(\llbracket t \rrbracket(a))\\
\llbracket t^{-1} \rrbracket(a) & = (\llbracket t \rrbracket(a))^{-1}\\
\llbracket t+r \rrbracket (a)& = \llbracket t \rrbracket(a)+\llbracket r\rrbracket(a)\\
\llbracket t \cdot r \rrbracket(a) & = \llbracket t \rrbracket(a) \cdot \llbracket r\rrbracket(a)\\
\llbracket e\colon \!\!\to t \rrbracket(a) & = \llbracket t \rrbracket(a), \mathit{if}\, \events \models \neg a \vee e = \top\\
\llbracket e\colon \!\!\to t\rrbracket(a) & = 0, \mathit{if}\, \events \models a \wedge e = \bot.
\end{align*}
\hrule
\caption{Definition of $\llbracket t \rrbracket(a)$ for $a \in \events_{\mathit{at}}$}
\label{FunCV}
\end{table}F
 
\subsection{Combining CVs with a probability function: expected values}
A $V_C$ expression, say $X$, denotes a value which is conditional on an event, 
that is it depends on the actual event $e$ chosen from $\events$. 
Therefore CVs are well-suited  for defining an expected value, denoted with $E_P(X)$.
The concept of an expectation lies at the basis of further definitions of 
probabilistic quantities such as variance, covariance, and correlation. 
Defining the expected value for a conditional value can be done if a besides a probability function, say $P$, $V_C$ expression 
in flat form is available, say $ \sum_{i= 1}^n e_i:\to v(t_i)$.
\[E_P(\sum_{i=1}^n e_i\!:\to v(t_i)) = \sum_{i =1}^n(  P(e_i))\cdot t_i).\]
These identities provide an axiom scheme for the function $E_P\!:CV \to V$.

Given a probability function structure  $\EPV(\events,\reals_0(\sg),\pf)$ and a CV structure
involving the same event space, say  
$\ECV(\events,\reals_{0}(\sg), \cv(\events,\reals_0(\sg)))$ a joint expansion exists.
Denoting the joint expansion with    
$\EPCV(\events,\reals_{0}(\sg), \cv(\events,\reals_0(\sg)),\pf)$ it can be further
expanded with an expected value operator named
$\widehat{E}_P$, interpreted  in compliance with the mentioned scheme, 
 to a structure $\EPCV(\events,\reals_{0}(\sg), \cv(\events,\reals_0(\sg)),\pf,\widehat{E}_P)$. Taken together for all event spaces $\events$
 and for all probability functions $\pf$
 the latter structures constitute a class of probability structures $K(\BA)$.

Instead of using an  axiom scheme,  a finite axiomatisation of $E_P(-)$ is given in Table~\ref{ExpVal}, 
from which each instance of the scheme can be derived. 
The equations (named $\EV\!_P$) of Table~\ref{ExpVal}  determine $E_P(-)$ on all $V_C$ expressions not involving
variables of sort $V_C$.

\begin{table}
\centering
\hrule
\begin{align}
E_P(X + Y) &= E_P(X) + E_P(Y)\\
E_P(x\!:\to v(y)) &= P(x) \cdot y
\end{align}
\hrule
\caption{$\EV_{\!P}$, axioms for the expected value operator, $x$ ranges over $E$, $y$ over $V$}
\label{ExpVal}
\end{table}

Grouping together the axioms collected thus far one finds an equational theory: 
$\MBPC_P = \BA+\Md+\DO+\Sign+\ABS+\PFBC_{\!P}+ \PFA_{P}+\UCV + \Sign_{cv} + \Md_{cv}+\Cond+\EV\!_P$ 
(meadow based probability calculus).
A plausible class of models for $\MBPC_{\!P}$ is 
$K(\BA)$. 
With a proof similar to that of Theorem~\ref{Compl:1}, it follows that $\MBPC_{\!P}$ is complete for such equations w.r.t. 
validity in $K(\BA)$.

$E_P$ can be eliminated from expressions of sort $V$ without free variables of sort $V_C$. 
Therefore an expression of sort $V$ without free variables of sort $V_C$  is provably  
equal within $\MBPC_P$ to an expression not involving subterms of sort $V_C$. 

\subsection{Variance, covariance, and correlation for conditional values}
\label{IDSQ}
On the basis of a definition of expectation, variance, covariance, 
and correlation  on conditional values can be introduced as derived operators as in~Table~\ref{More}.
\begin{table}
\centering
\hrule
\begin{align}
\mathit{VAR_P}(X) &= E_P(X^2) - (E_P(X))^2  \\
\mathit{COV}\!_P(X,Y) &= E_P(X \cdot Y) - E_P(X)\cdot E_P(Y)  \\
\mathit{CORR}^{sq}_P(X,Y) &= 
\frac{\mathit{COV}\!_P(X,Y)^2}{\mathit{VAR}_P(X) \cdot \mathit{VAR}_P(Y)} 
\end{align}
\hrule
\caption{$\EV_{\!P}$, axioms for variance, covariance, and correlation for conditional values}
\label{More}
\end{table}

Let $X$ and $Y$ 
be $V_C$ expressions with flat forms $X= \sum_{i= 1}^n e_i:\to v(t_i)$ and 
$Y= \sum_{i= 1}^m f_i:\to v(r_i).$ The equations in Table~\ref{ExpVal} provide explicit definitions  of variance, covariance, and 
correlation for $X$, resp. $Y$.
 
There is no novelty to these definitions except for the effort made to make each definition  fit a framework that 
has been setup on the basis of an algebraic specification. By proceeding in this manner an axiomatic framework 
is obtained for equational reasoning about each of these technical notions.

Forgetting the subscript for $E_P$, that is using $E(X)$ instead of $E_P(X)$, and similarly for the other operators, 
is common practice in probability theory. Doing so, however requires that it is 
apparent from the context which probability function is used. Moreover it must be assumed that for $X$ and for $Y$ 
the same probability function applies.

 \subsection{Extracting a probability mass function from a conditional value}
 Given a conditional value  $X= \sum_{i= 1}^{n} e_i\!:\to v(t_i)$ in non-overlapping flat form, 
 and a probability function $P$ 
 the probability mass function,  $\lambda x. P(X=x)$ for $X$ is supposed to yield for each value $x$ the 
 probability that $X$ 
takes value $x$.
  An explicit definition for the PMF of $X$ is as follows:
 \[\mathit{Pmf}_{\!P}(X) =L\, x_{ \in V}. \sum_{i = 1}^n(0(t_i-x) \cdot P(e_i)). \]
This specification of $\mathit{Pmf}_{\!P}$ is schematic and for that reason does not 
achieve the simplicity found for the expected value operation.
\begin{problem}
Can $\mathit{Pmf}_{\!P}$ be specified by means of a fixed and finite number of equations rather than with an axiom scheme
involving an equation for each non-overlapping closed $V_C$ expression?
\end{problem}
\begin{proposition}
Equivalence of definitions for expectation and  variance for CV expressions in 
non-overlapping flat form via (joint) PMFs extraction.

\begin{enumerate}
 \item $E_P(X) = E_{\mathit{pmf}}(\mathit{Pmf}_{\!P}(X)),$
\item
 $\mathit{VAR}_P(X) = \mathit{VAR}_{\mathit{pmf}}(\mathit{Pmf}_{\!P}(X)).$
\end{enumerate}
 \end{proposition}
 
 \label{PMFext}
  \begin{proof} Let $ X= \sum_{i= 1}^{n} e_i\!:\to v(t_i)$ be a non-overlapping flat $V_C$ expression. 
  Making use of the facts listed in Paragraph~\ref{FSS}, one obtains: 
  $E_{pmf}(L\, x.P(X = x)) = 
    \sum_x^\star\sum_{i = 1}^n(0(t_i-x) \cdot P(e_i))=\\
  \sum_{i = 1}^n\sum_x^\star(0(t_i-x) \cdot P(e_i))=\sum_{i = 1}^n \sum_x^\star(0(t_i-x) \cdot P(e_i))=
  \sum_{i = 1}^n(t_i \cdot P(e_i))= E_P(X).$
 \end{proof}

 \subsection{Joint PMF extraction for event sharing conditional values}
 Two conditional values are event sharing if both have conditions over the same domain. 
 Extraction of a joint PMF  from event sharing conditional values works as follows.
 Given two $V_C$ expressions $X$ and $Y$ with similar nonoverlapping flat forms $\sum_{i= 1}^{n}(e_i\!:\to t_i)$ and 
  $\sum_{i= 1}^{n}(e_i\!:\to r_i)$  the joint PMF for these conditional values, denoted by $P(X = x,Y=y)$, is defined by
\[P(X=x,Y=y)  = \sum_{i= 1}^{n}(0(t_i-x) 	\cdot 0(r_i-y)\cdot P(e_i)).\]

Extending Proposition~\ref{PMFext} the following connections between definitions involving a conditional value and 
definitions involving a PMF or a joint PMF can be found.
\begin{proposition} Equivalence of definitions for covariance and correlation (squared) via CVs and via (joint) PMFs. 
\begin{enumerate}
 \item
$ \mathit{COV}_{\!P}(X,Y) = \mathit{COV}_{\!\mathit{pmf}}(L\, x,y.P(X=x,Y=y)),$
\item
 $ CORR^{sq}_P(X,Y) = CORR_{\mathit{pmf}}^{sq}(L\, x,y.P(X=x,Y=y)).$
\end{enumerate}
\end{proposition}

\section{The multidimensional case}
\label{MDC}
In the multidimensional case the event space is considered a  product of event spaces. 
In the multi-dimensional case CVs  occurring in a vector of CVs are supposed by default not to be event space sharing and 
the notion of a joint probability function working over a tuple of event spaces enters the picture.

The multi-dimensional case becomes relevant once tuples (vectors) of CVs are considered in combination with 
a plurality of joint probability functions for product spaces of higher dimensional event space corresponding to 
various vectors of CVs such that there may not exist a joint probability function for the 
full product space.

\subsection{Multidimensional  probability functions}
Let $D= \{a_1,\ldots,a_n\}$ be a finite set. The elements of $D$ will be called dimensions. $D$ is called a dimension set, and 
it is assumed that $n = \#(D)$.
\begin{definition} (Arities over $D$) $ar_D$, the collection of arities over dimension set $D$,  denotes the 
set of finite non-empty sequences of elements of $D$ without repetition. 
\end{definition}
Elements of $ar_D$ will serve as arities of probability functions on multi-dimensional event spaces. 
$l(w)$ denotes the length of $w \in ar_D$. 
\begin{definition} (Arity family)
Given an event space $E$, and a name $P$ for a probability function, an arity family (for $E$ and $P$)  is a 
finite subset $W$ of $ar_D$ which is (i) closed under permutation, and (ii) closed under taking non-empty subsequences, 
and (iii) which contains for each $d \in D$ the arity $(d)$, that is the one-dimensional arity consisting of dimension $d$ only.
\end{definition} 
For each dimension $d \in D$ the presence of a sort  $E_d$ of events for dimension $d$ is assumed. For simplicity of notation it is assumed that these sorts are identical, so that only a sort $E$ is required.

\begin{definition}
A probability family (denoted $\PFF_W$)  for an arity family $W \subseteq ar_D$ consists of a 
probability function $P_w \colon E^{l(w)} \to V$ for each  $w \in W$,  such that for all $w \in W$ each  the axioms in 
Table~\ref{PFF} (taken from~\cite{BP2016}) are satisfied. \end{definition}
The axioms of Table~\ref{PFF} case correspond to the axioms for a probability 
function of Table~\ref{PF} in the one dimensional case.

Because in an arity repetition of dimensions is disallowed these axioms reduce to 
what we had already in the case of a single dimension. 

\begin{table}
\centering
\hrule
\begin{align}
\label{eq:wperm}
P^{d,u,e,u^\prime}(y_1,x_1\ldots,x_{l},y_2,z_{1},\dots,z_{l^\prime}) &= P^{e,u,d,u^\prime}(y_2,x_1\ldots,x_{l},y_1,z_{1},\dots,z_{l^\prime}) \\
\label{eq:23:w}
P^{d}(\top) &=1\\
P^{d}(\bot) &=0\\
\label{eq:25:w}
P^{d,w}(\top,x_1,\ldots,x_{n}) &= P^{w}(x_1,\ldots,x_{n})\\
\label{eq:24:w}
P^{d,w}(\bot,x_1,\ldots,x_{n}) &=0\\
\label{eq:25:w}
P^w(x_1,\ldots,x_{n})  & = |P^w(x_1,\ldots,x_{n})|\\
\label{eq:26:w}
P^{d,u}(x \vee y,x_1,\ldots,x_{l}) &= P^{d,u}(x,x_1,\ldots,x_{l})+P^{d,u}(y,x_1,\ldots,x_{l}) \nonumber \\
&\quad -  P^{d,u}(x \wedge y,x_1,\ldots,x_{l})  
\end{align}
\hrule
\caption{$\PFF_{W,P}$:  axioms for a probability  function family with name $P$ 
(with $d,e \in D$, $w, (d,u), (e,u,d,u^\prime) \in W, n = l(w)$, and 
$ u, u^\prime \in ar_D \cup \{\epsilon\}, l = l(u), l^\prime = l(u^\prime).$}
\label{PFF}
\end{table}

\subsection{Multivariate conditional values}
Just as in the one-dimensional case, multivariate conditional values are the elements of sort $V_C$.
$V_C$ has, besides the embedding $v$ from $V$ into $V_C$ (which must meet the requirements of Table~\ref{UCV}),
for each $d \in D$ a constructor $-: \to_d -$ of type $E \times V_C \to V_C$. $-: \to_d -$ must satisfy the 
requirements $\Cond_d$ which result from $\Cond$ in Table~\ref{CondEqs} by replacing operator $-:\to-$ by $-:\to_d-$ in all equations.
In addition to these requirements the equations $\Cond_{mv}$ of Table~\ref{CondMv} must be satisfied for all different pairs $a,b \in D$.

\begin{table}
\centering
\hrule
\begin{align}
e:\to_a (f :\to_b X) = f:\to_b (e :\to_a X)
\end{align}
\hrule
\caption{$\Cond_{mv}$: commuting multivariate condition constructors}
\label{CondMv}
\end{table}

\subsection{Expected value operators}
For a specification of the expected value operator it is assumed that $d_1,\ldots,d_n$ is an enumeration without repetitions of $D$.
For each $w \in W$ a separate expected value operator $E^w_{P}$ arises. Each operator is specified by means of two equations
as displayed in Table~\ref{ExpValMv}.
\begin{table}
\centering
\hrule
\begin{align}
E^w_{P}(X + Y) &= E^w_{P}(X) + E^w_{P}(Y)\\
E^w_{P}(x_1\!:\to_{d_1}(\ldots (x_n\!:\to_{d_n} v(y)\dots)) &= P_w(x_1,\ldots,x_n) \cdot y
\end{align}
\hrule
\caption{$\EV_{\!P,w}$, axioms for the expected value operator for arity $w$}
\label{ExpValMv}
\end{table}

Given the multi dimensional expected value operator, corresponding operators for 
variance, covariance, and correlation can be derived un the usual manner.

\subsection{Summing up}
Collecting the equations mentioned thus far for the multidimensional setting 
the axiom system 
 $\MBPC^W_P = \BA+\Md+\DO+\Sign+\ABS+ \UCV + \Sign_{cv} + \Md_{cv}+\Cond_{d\, (d \in D)} +\PFF_{W,P} + \EV_{\!P,w (w \in W)}$ is obtained.
 
 Completeness of these axiomatisations can be shown with the same methods as for the 1D case.
The design of these structures can be somewhat simplified if for each subset of $D$ at most a single probability function is admitted,
having the arguments for the different dimensions in a fixed order. When adopting this alternative, Table~\ref{PFF} needs to be redesigned as follows: permutation axioms are dropped and axioms involving the first argument must be replicated for each argument position.

\section{Concluding remarks}
This paper is a sequel to~\cite{BP2016} where a meadow based approach to the equational 
specification of probability functions was proposed. In~\cite{BM2015} probabilistic choice is 
formalised with the meadow of reals as a number system. The equations in that paper demonstrate, 
just as well as the equations in Table~\ref{PF}, an attractive compatibility between the requirements of 
probability calculus and the treatment of division in a meadow.

In~\cite{ShaferV2006} an extensive survey is presented of the history leading up to Kolmogorov's choice of axioms, 
and to Kolmogorov's claim that these axioms are what probability is about. The equations in $\PFBC_{\!P}+ \PFA_P$ do not take the 6th axiom into account, however, which 
asserts that if  $(e_i)_{i \in \nats}$ is an infinite descending chain of 
events such that only $\bot$ is below each element of the chain, then 
$\lim_{i \to \infty}P(e_i) = 0$. A closer resemblance with Kolmogorov's original axioms is found if the
equation in Table~\ref{PF} is replaced by the conditional equation $e \wedge f  = \bot \to P(e \vee f) = P(e) + P(v)$. 
This replacement produces a logically equivalent axiom system. The equation of Table~\ref{PF} is 
preferred because it is logically simpler than a conditional equation.

Conditional values
play the role of a discrete random variables with finite range. 
By working with conditional values the use of a sample space underlying the event space is avoided 
which helps to maintain the style and simplicity of the axiomatisation of probability functions of~\cite{BP2016}.
Instead of including an additional sort $V_C$, the conditional values might be viewed as an extension of the sort $V.$
A reason for not doing so, however, is to prevent $P$ from taking values of  the form say $P(e) = f\!:\to v(1/2)$.

Regarding the choice of terminology, 
the presence of alternative options must be mentioned, for instance in~\cite{BertsekasT2008}
a probability function is referred to as a probability law.

As a technical tool finite support summation is introduced, a novel binding operator on meadows. 
Finite support summantion is of independent interest for the theory of meadows and it gives rise to intriguing new questions. 
Further it is worth mentioning that working with $1/0 = 0$ 
in matrix theory is pursued in e.g.~\cite{MatsuuraS2016}.

For the derived operations $1(-)$ and $0(-)$ of Table~\ref{DO}
the original notation from~\cite{BBP2013,BBP2015} is $1(x) =1_x$, resp. $0(x)=0_x$, 
which notations may still be used as alternatives. The chosen notation is preferable if a sizeable expression is substituted for $x$.
Table~\ref{Md} makes use of inversive notation.
  The phrase ``inversive notation'' was coined in \cite{BM2011a} where it stands 
  in contrast with ``divisive notation'' which involves a two place division operator symbol. In~\cite{BM2011a} the equivalence 
  of both notations is discussed.
  Two place division   is provided as a derived operation in Table~\ref{DO}. Division commonly
  appears  in a plurality of syntactical forms: $ x {:} y, x/y,\sfrac{x}{y}$, and $\frac{x}{y}$.
  These diverse forms are not in need of a separate defining equation, just as much as in the specification of a 
  meadow no mention is made of the existing notational variation for multiplication (viz. $x \times y, x \cdot y, x.y$ and $xy$).

\paragraph{Acknowedgement} Yoram Hirschfeld, Kees Middelburg and Alban Ponse gave useful comments on a previous version of the paper.

\appendix
\section{Random variables}
\label{MVF}
The notion of a random variable plays a central role in many presentations of probability theory. 
In the presentation of the current paper the role of random variables is played by a conditional values (CVs) instead. 
In this Appendix it will be outlined how to view a CV as a random variable provided that the event space is finite. 

\subsection{From implicit sample space to explicit sample space}
Given event space $\events$, the subset of its domain  $\events_{\mathit{at}}$ consisting of atoms 
as defined in Paragraph~\ref{MofCV} can be taken for the corresponding sample space and then
a random variable is supposed to be a function from sample space to values. 
Viewing $\events_{\mathit{at}}$ as a sample space, for each close conditional value expression $X$, the function
 $\llbracket X \rrbracket$, as specified in Table~\ref{FunCV}, qualifies as a
random variable. 

I prefer not to have $\events_{\mathit{at}}$ as a sort because the resulting setting with 
$\events_{\mathit{at}}$ as a subsort of $\events$
  is not easily  reconciled with equational logic. 
  Logical difficulties with the equational logic of subsorts persist  in spite of the presence of  many works that have been 
  devoted to that particular complication.

Now summation over the sample space $\events_{\mathit{at}}$ is specified as follows.
For an event space $\events$ and a term $t$ of sort $V$, then $\sum_{\alpha \in \events_{\mathit{at}}}^\star t=0$   if there 
are either none or infinitely many atomic events in $|\events|$ and otherwise 
\[\sum_{\alpha \in \events_{\mathit{at}}}^\star t=[a_i/\alpha]t+ \ldots + [a_k/\alpha]t\] with 
$a_1,\ldots,a_k$ an enumeration without repetitions of the atomic events of $\events$.
 Provided $\events$ is finite, the expectation of  $\llbracket X \rrbracket$ can be defined by summation over the sample space, using an identity which lies outside first order equational logic:
\[E_P(\llbracket X \rrbracket) = \sum_{\alpha \in \events_{\mathit{at}}}^\star ( \llbracket X \rrbracket(\alpha) \cdot P(\alpha))\]
\subsection{Random variables in colloquial language}
Random variables play a key role in many accounts of probability theory. 
However, the concept of a random variable seems to be rather informal and its use is often cast in colloquial language. 
A common wording states that  
 ``a random variable is the outcome of a stochastic process''. 
 Complicating an understanding of a random variable, however, is the fact that the mathematical definition of it, 
which reads ``a function from sample space to reals'' makes no reference to any variable or variable name, 
or to a probability function, or to a stochastic mechanism. 
In~\cite{W2016} it is asserted about a random variable that it is: 
\begin{quote}
... a variable whose value is subject to variations due to chance (i.e. randomness, in a mathematical sense).... 
A random variable can take on a set of possible different values (similarly to other mathematical variables), 
each with an associated probability, in contrast to other mathematical variables.'
\end{quote}
In~\cite{KA2016} a random variable is explained as a mapping from ``outcomes'' to values which provides quantification, while
 the main argument put forward for the introduction of a random variable is about the use of its name,
 and at the same time the suggestion is made that a 
 random variable is linked to a probability function.
In~\cite{BertsekasT2008}  it is stated that
\begin{quote} A discrete random variable has an associated probability mass function ..
\end{quote}
In the introductory probability refresher of~\cite{Barber2012} the domain of a variable is said to be the set of states it can take, 
while the relation between (random) variables and events is explained as follows:
\begin{quote} For our purposes, events are expressions about random variables, such as \emph{Two heads in 6 coin tosses}.
\end{quote}

\end{document}